\documentclass[reqno]{amsart}
\usepackage{hyperref}

\usepackage[]{graphicx}

\usepackage{amsthm}
\usepackage{amsmath}
\usepackage{latexsym}
\usepackage{amsfonts}
\usepackage{amssymb}

\newcommand{\CC}{{\mathbb C}}
\newcommand{\RR}{{\mathbb R}}
\newcommand{\NN}{{\mathbb N}}

\newcommand{\DD}{{\mathbb D}}
\newcommand{\JJ}{{\mathbb J}}

 \newcommand{\ro}{\varrho}
\newcommand{\e}{\varepsilon}
\newcommand{\be}{\beta}
\newcommand{\al}{\alpha}
 \newcommand{\De}{\Delta}
\newcommand{\la}{\lambda}
\newcommand{\g}{\gamma}
\newcommand{\de}{\delta}
\newcommand{\ph}{\varphi}
\newcommand{\Ga}{\Gamma}
\newcommand{\Om}{\Omega}

\newcommand{\bin}[2]{
  \left(
     \begin{array}{@{}c@{}}
         #1 \\ #2
      \end{array}
   \right)}

\newcommand{\Si}{\Sigma}
\newcommand{\CMF}{\mathcal{CMF}}
\newcommand{\BBF}{\mathcal{BF}}

\begin{document}

\title[Completely monotone functions...]
{Completely monotone functions and some classes of fractional evolution equations}

\author
{Emilia Bazhlekova}

\date{February 12, 2015}

\address{
Institute of Mathematics and
Informatics, Bulgarian Academy of Sciences}

\email{e.bazhlekova@math.bas.bg}

\maketitle

\begin{abstract}
			The abstract Cauchy problem for the distributed order fractional evolution equation in the Caputo and in the Riemann-Liouville sense is studied for operators generating a strongly continuous one-parameter semigroup on a Banach space. Continuous as well as discrete distribution of fractional time-derivatives of order less than one are considered.  
		The problem is reformulated as an abstract Volterra integral equation. It is proven that its kernel satisfy certain complete monotonicity properties. Based on these properties, the well-posedness of the problem is established and a series expansion of the solution is obtained. In case of ordered Banach space this representation implies positivity of the solution operator. In addition, a subordination formula is obtained. 
		

 \end{abstract}

\numberwithin{equation}{section}
\newtheorem{theorem}{Theorem}[section]
\newtheorem{lemma}[theorem]{Lemma}
\newtheorem{example}[theorem]{Example}
\newtheorem{corollary}[theorem]{Corollary}
\newtheorem{prop}[theorem]{Proposition}
\newtheorem{definition}[theorem]{Definition}
\newtheorem{remark}[theorem]{Remark}

\section{Introduction}
The evolution equations of fractional order are extensively used for modeling of materials and processes with memory. 
Recently, in the attempt to find more adequate models, equations involving discrete or continuous distribution of fractional derivatives are introduced. 

In this paper, we consider the fractional evolution equation of distributed order in the following two alternative forms:
\begin{equation}\label{C}
\int_0^1 \mu(\be) ^C\! D_t^\be  u(t)\,d\be=Au(t), \ \ t>0,
\end{equation}
and
\begin{equation}\label{RL}
u'(t)=\int_0^1 \mu(\be) D_t^{\be}A u(t)\,d\be, \ \ t>0,
\end{equation} 
where $^C\! D_t^\be$ and $ D_t^\be$ are the fractional time-derivatives in the Caputo and in the Riemann-Liouville sense, respectively, and
 $A$ is a closed linear unbounded operator densely defined in a Banach space $X$. The initial condition $ u(0)=a\in X$ is prescribed. For the weight function $\mu$ two cases are considered: 
\begin{itemize}
\item discrete distribution
\begin{equation}\label{case1}
\mu(\be)=\de(\be-\al)+\sum_{j=1}^m b_j \de(\be-\al_j),
\end{equation}
 where $ 1> \al>\al_1...>\al_m> 0, \ b_j> 0, \ j=1,...,m,\ m\ge 0$,  and $\de$ is the Dirac delta function; 
\item  continuous distribution
\begin{equation}\label{case2}
\mu\in C[0,1], \ \mu(\be)\ge 0,\ \be\in[0,1],
\end{equation}
and $\mu(\be)\neq 0$ on a set of a positive measure.
\end{itemize}

In the case of discrete distribution, equations (\ref{C}) and (\ref{RL}) are reduced to the multi-term time-fractional equations
\begin{equation}\label{multitermC}
^C\! D_t^{\al}u(t)+\sum_{j=1}^m {b_j}\, {^C\!} D_t^{\al_j}u(t)=Au(t), \  \  t>0,
\end{equation} 
and 
\begin{equation}\label{multitermRL}
u'(t)=D_t^{\al}Au(t)+\sum_{j=1}^m {b_j}\,  D_t^{\al_j}Au(t),\  \  t>0,
\end{equation} 
respectively.   
Note that if $m=0$ (single-term fractional evolution equation) 
problem (\ref{multitermC}) is equivalent to (\ref{multitermRL}) with $\al$ replaced by $1-\al$. However, in general, similar equivalence does not hold for equations 
(\ref{C}) and (\ref{RL}).

Problems (\ref{C}) and (\ref{RL}) have found numerous applications in the modeling of anomalous relaxation and diffusion phenomena, see e.g. 
\cite{Chechkin, MainardiRel, MainardiDiff, Sokolov}. This leads 
to a growing interest to such problems in the last years.  
Concerning the equation in the Caputo sense (\ref{C}), we refer to \cite{VarshaJMAA, Burrage, Jin, KosticFCAA, Yamamoto, LuchkoMultiterm} for the study of its multi-term variant (\ref{multitermC}) and to \cite{Chechkin, Mirjana, K, LuchkoFCAA, Sokolov} for the case of continuous distribution.  
Concerning the equation in the Riemann-Liouville sense (\ref{RL}), two- and three-term variants of problem (\ref{multitermRL}) are considered in \cite{
Guo, Langlands, Pagnini} and the continuous distribution case of (\ref{RL}) is discussed in \cite{MainardiRel, MainardiDiff, Sokolov}. Although problem (\ref{C}) is already studied for various classes of operators $A$: the Laplace operator in different settings \cite{K, Jin}, the Riesz space-fractional derivative \cite{Burrage}, second order symmetric uniformly elliptic operators \cite{K, Yamamoto, LuchkoFCAA, LuchkoMultiterm}, general operators in a Banach space setting \cite{KosticFCAA}, etc., to the best of the author's knowledge, the mathematical study of problem (\ref{RL}) is still quite limited. 

In this paper, it is assumed that the operator $A$ is a generator of a $C_0$-semigroup (see e.g. \cite{Engel}), i.e. that the classical abstract Cauchy problem 
\begin{equation}\label{classical}
 u'(t)=Au(t),\ \ t>0;\ \ \ \ u(0)=a\in X,
\end{equation}
is well-posed. Reformulating  problems (\ref{C}) and (\ref{RL}) as abstract Volterra integral equations, 
we propose a unified approach to their study. 
We prove that the scalar kernels 
of the corresponding integral equations have certain complete monotonicity properties 
 and derive useful consequences for the original equations (\ref{C}) and (\ref{RL}) based mainly on these properties. 

Recall that a function $f:(0,\infty)\to\RR$ is called completely monotone if it is of class $C^\infty$ and
\begin{equation*}
(-1)^n f^{(n)}(t)\ge 0, \mbox{\ for\ all\ } t> 0, \ n=0,1,...
\end{equation*} 
Completely monotone functions appear naturally in the models of relaxation and diffusion processes. Such are the exponential and the Mittag-Leffler functions of negative argument, which are obtained as solutions of the classical and the single-term fractional relaxation equations. 
  For details and references on the complete monotonicity property see e.g. \cite{MillerSamko} and \cite{Pruss}, Chapter 4. Concerning Mittag-Leffler-type functions we refer to \cite{MillerITSF, KST, bookML}.
	
The present paper is organized as follows. Section 2 contains preliminaries. In Section 3 problems (\ref{C}) and (\ref{RL}) are rewritten in equivalent form as abstract Volterra integral equations and the properties of their scalar kernels are studied. 
In Section 4 the well-posedness of the problems is proven  and 
series expansion of the solutions is obtained in terms of the resolvent operator of $A$. 
In addition, the positivity of the solution operators is discussed.  In Section 5, a subordination formula is obtained. 




\section{Preliminaries}  
 The sets of positive integers, real and complex numbers are denoted as usual by ${\mathbb N},\ {\mathbb R},$ and $ {\mathbb C}$, respectively, and ${\mathbb N}_0={\mathbb N}\cup \{0\}$, $\RR_+=[0,\infty)$.  
Denote by $\Sigma_\theta$ the sector  
\begin{equation*}
  \Sigma_\theta:=\{s\in\mathbb{C};\ s\neq 0, |\arg s|<\theta\}.
	\end{equation*}
	Denote by $\ast$ the Laplace convolution:
$$(f_1\ast f_2)(t)=\int_0^t f_1(t-\tau)f_2(\tau)\, d\tau.$$

Let $J_t^\g$ be the fractional Riemann-Liouville integral $$ J_t^\g f(t)=\frac{1}{\Gamma(\g)}\int_0^t (t-\tau)^{\g-1}f(\tau)\,d\tau=\left\{\frac{t^{\g-1}}{\Gamma(\g)}\right\}\ast f(t),\ \ \g>0,$$
where $\Ga(\cdot)$ is the Gamma function. 

The Caputo and the Riemann-Liouville fractional derivatives of order $\be\in[0,1]$, $^C\! D_t^\be$ and $ D_t^\be$, 
are defined by $^C\! D_t^0= { D_t^0}=I$, $^C\! D_t^1= { D_t^1}=d/dt$, and
\begin{equation*}
 ^C\! D_t^\be =  J_t^{1-\be}  {D_t^1},\ \   D_t^\be = {D_t^1} J_t^{1-\be},\ \ \ \be\in(0,1).
\end{equation*} 
Application of the Laplace transform  $$\mathcal{L}\{f(t)\}(s)=\widehat{f}(s)=\int_0^\infty e^{-st} f(t)\, dt$$
 to the operators of fractional integration and differentiation gives \cite{KST}
\begin{eqnarray}
&&\mathcal{L}\{J_t^\g f\}(s)=s^{-\g} \widehat{f}(s),\ \ \g>0,\nonumber\\
&& \mathcal{L}\{{^C\! D}_t^\be f\}(s)=s^\be \widehat{f}(s)-s^{\be-1}f(0+),\ \ \be\in(0,1),\label{Lap}\\
&&\mathcal{L}\{D_t^\be f\}(s)=s^\be \widehat{f}(s)-(J_t^{1-\be} f)(0+),\ \ \be\in(0,1).\nonumber  
\end{eqnarray}
Denote as usual by $E_{\al,\be}(\cdot)$ the Mittag-Leffler function 
\begin{equation*}
E_{\al,\be}(z)=\sum_{k=0}^{\infty}\frac{z^k}{\Ga(\al k+\be)},
\ \ \al,\be, z\in \CC,\ \Re \al>0.
\end{equation*}
Recall the Laplace transform pairs ($\Re\al>0, \ \mu\in\RR, \ t>0$)
\begin{equation}\label{pairs}
\mathcal{L}\left\{\frac{t^{\al-1}}{\Ga(\al)}\right\}={s^{-\al}},\ \ \ \ \mathcal{L}\left\{t^{\be-1}E_{\al,\be}(-\mu t^\al)\right\}=\frac{s^{\al-\be}}{s^\al+\mu}.
\end{equation}


The characterization of completely monotone functions is given by the Bernstein's theorem (see e.g. \cite{Feller}) which states that 
a function $f:(0,\infty)\to\RR$ is completely monotone if and only if it can be represented as the Laplace transform of a nonnegative measure.

A $C^\infty$ function $f:(0,\infty)\to\RR$ is called a Bernstein function if it is nonnegative and its first derivative $f'(t)$ is a completely monotone function. 

The classes of completely monotone functions and Bernstein functions will be denoted by $\CMF$ and $\BBF$.
Some properties of these classes of functions, which will be used further, are summarized in the next proposition.

\begin{prop} The following properties are satisfied:\\
(a) The class $\CMF$ is closed under pointwise addition and multiplication;\\
(b) If $f\in\CMF$ and $\ph\in\BBF$, then the composite function $f( \ph)\in\CMF$;\\
(c) If $\ph\in\BBF$, then $\ph(s)/s\in\CMF$;\\
(d) Let $f\in L^1_{loc}(\RR_+)$ be a nonnegative and nonincreasing function, such that $\lim_{t\to +\infty}f(t)=0$. Then $\ph(s)=s\widehat {f}(s)\in\BBF$;\\
(e) If $f\in L^1_{loc}(\RR_+)$ and $f$ is completely monotone, then $\widehat{f}(s)$ admits analytic extension to $\Sigma_\pi$ and
$|\arg\widehat{f}(s)|\le |\arg s|,\ \ s\in \Sigma_\pi$.
\end{prop}
\begin{proof}
Properties (a) and (b) follow by application of the product and the chain rule of differentiation (a detailed proof of (b) can be found in \cite{MillerSamko}). For (c) and (d) see the proof of Proposition 4.3 in \cite{Pruss}. For (e) see \cite{Pruss}, Example 2.2.
\end{proof}

Let $X$ be a Banach space with norm $\|.\|$. Let $A$ be a closed linear unbounded operator in $X$ with dense domain $D(A)$, equipped with the graph norm $\|.\|_A$, $\|x\|_A:=\|x\|+\|Ax\|.$
Denote by $\varrho(A)$ the resolvent set of $A$ and by $R(s,A)$ the resolvent operator of $A$: $R(s,A)=(s-A)^{-1}$, $s\in \varrho(A)$.

Next we recall some definitions and basic theorems, given in \cite{Pruss},
  concerning the abstract Volterra integral equation
 \begin{equation}\label{V}
u(t)=a+\int_0^t k(t-\tau)A u(\tau)\, d\tau, \ \ t\ge 0;\ \ \ \  a\in X,
\end{equation}
with a scalar kernel $k\in L^1_{loc} (\RR_+)$. 

A function $u\in C(\RR_+;X)$ is called a strong solution of (\ref{V}) 
if $u\in C(\RR_+;D(A))$ and (\ref{V}) holds on $\RR_+$.

Problem (\ref{V}) is said to be well-posed if for each $a\in D(A)$ there is a unique strong solution $u(t;a)$ of (\ref{V}) and $a_n\in D(A),\ a_n\to 0$ imply $u(t;a_n)\to 0$ in $X$, uniformly on compact intervals.

For a well-posed problem the solution operator $S(t)$ is defined as usual by
$$
S(t)a=u(t;a),\ \ a\in D(A),\ t\ge 0.
$$
In this paper it is assumed that $A$ is a generator of a bounded $C_0$-semigroup $T(t)$, $\|T(t)\|\le M$, $t\ge 0$. The Hille-Yosida theorem (see e.g. \cite{Engel}) then implies that $(0,\infty)\subset\ro(A)$ and
\begin{equation}\label{HY}
\|R(s,A)^n\|\le M/s^n,\ \ s>0,\ n\in\NN.
\end{equation}

Suppose $\int_0^\infty e^{-st}|k(t)|\, dt<\infty$ for $s>0$ and $\widehat{k}(s)\neq 0$, $1/\widehat{k}(s)\in \ro(A)$ for $s>0$.
Then the Laplace transform of the solution operator $S(t)$ of problem (\ref{V}) 
$$
H(s)=\int_0^\infty e^{-st}S(t)\,dt, \ \ \Re s>0
$$
is given by
\begin{equation}\label{H}
H(s)=\frac{g(s)}{s}R(g(s),A),\ \ g(s)=1/\widehat{k}(s).
\end{equation}
The Generation Theorem (\cite{Pruss}, Theorem 1.3) states that problem (\ref{V}) is well-posed with solution operator $S(t)$ satisfying  $\|S(t)\|\le M,\ t\ge 0,$ 
if and only if  
\begin{equation}\label{GenTh}
\|H^{(n)}(s)\|\le M\frac{n!}{s^{n+1}},\ \ \mbox{ for\ all\ } s>0,\ n\in\NN_0.
\end{equation}


\section{Integral reformulation and properties of the kernels}
  
	We reformulate problems (\ref{C}) and (\ref{RL}) as Volterra integral equations of the form (\ref{V}) with appropriate kernels $k(t)$. 
	 By applying (formally) the Laplace transform and, by the use of (\ref{Lap}), it follows for the solution of (\ref{V}) 
\begin{equation}\label{solV}
\widehat{u}(s)=\frac{1}{s}(1-\widehat{k}(s) A)^{-1}a 
\end{equation}
and for the solutions of problems (\ref{C}) and (\ref{RL}), respectively
\begin{equation}\label{sol}
\widehat{u}(s)= \frac{h(s)}{s}\left(h(s)- A\right)^{-1}a,\  \widehat{u}(s)=\frac{1}{s}\left(1-\frac{h(s)}{s} A\right)^{-1}a,
\end{equation}
where 
\begin{equation}\label{h}
h(s)=\int_0^1 \mu(\be)s^\be\, d\be.
\end{equation}
Note that in the discrete distribution case (\ref{case1}) $h(s)$ admits the representation
\begin{equation}\label{hmt}
h(s)=
s^\al+\sum_{j=1}^m b_j s^{\al_j}.
\end{equation}
Comparing (\ref{sol}) to (\ref{solV}), it follows for the kernels $k_1(t)$ and $k_2(t)$, corresponding to problems (\ref{C}) and (\ref{RL}), respectively:
\begin{equation}\label{ker}
\widehat{k_1}(s)=\left(h(s)\right)^{-1},\  \widehat{k_2}(s)=h(s)/s.
\end{equation}
Representation (\ref{H}) suggests that it is convenient to define also the functions: 
$$
g_i(s)=1/\widehat{k_i}(s),\ \ i=1,2,
$$
that is
\begin{equation}\label{g}
g_1(s)=h(s),\ g_2(s)=s/h(s),
\end{equation}
where $h(s)$ is defined in (\ref{h}).
Some useful complete monotonicity properties of the functions $k_i(t)$ and $g_i(s)$, $i=1,2,$ are proven in the next theorem.

\begin{theorem} 
Let $\mu(\be)$ be either of the form (\ref{case1}) or of the form (\ref{case2}) with the additional assumptions $\mu\in C^3[0,1],$ $\mu(1)\neq 0$, and $\mu(0)\neq 0$ or $\mu(\be)=a\be^\nu$ as $\be\to 0$, where $a,\nu>0$.
Then the functions $k_i(t)$ and $g_i(s)$, $i=1,2,$ have the following properties:\\
(a) $k_i \in L^1_{loc}(\RR_+)$ and $\lim_{t\to +\infty}k_i(t)=0$;\\
(b) $k_i(t)\in\CMF$ for $t> 0$; \\
(c) $k_1\ast k_2\equiv 1$;\\
(d) $g_i(s)\in\BBF$ for $s> 0$;\\
(e) $g_i(s)/s\in\CMF$ for $s> 0$;\\
(f) $g_i(s)$ admits analytic extension to $\Sigma_\pi$ and for any $s\in \Sigma_\pi$ $$|\arg g_i(s)|\le |\arg s|.$$ In the multi-term case (\ref{hmt}) 
 a stronger inequality holds: $$|\arg g_i(s)|\le \al |\arg s|,\ s\in \Sigma_\pi.$$
 \end{theorem}

Let us first consider some particular cases. Applying (\ref{pairs}), (\ref{hmt}), (\ref{ker}) and (\ref{g}),  it follows in the single-term case ((\ref{hmt}) with $m=0$): 
$$
 k_1(t)=\frac{t^{\al-1}}{\Ga(\al)},\ k_2(t)=\frac{t^{-\al}}{\Ga(1-\al)}, \ \ \ g_1(s)=s^\al,\ g_2(s)=s^{1-\al},
$$
and in the double-term case ((\ref{hmt}) with $m=1$):
\begin{eqnarray}
&& k_1(t)=t^{\al-1}E_{\al-\al_1,\al}(-b_1 t^{\al-\al_1}),\ k_2(t)=\frac{t^{-\al}}{\Ga(1-\al)}+b_1\frac{t^{-\al_1}}{\Ga(1-\al_1)},\nonumber\\
&& g_1(s)=s^\al+b_1 s^{\al_1},\ g_2(s)=\frac{s}{s^\al+b_1 s^{\al_1}}.\nonumber
\end{eqnarray}
Thus, in the single-term case Theorem 3.1 is straightforward. 
In the double term case, statements (a) and (b) are trivial for $k_2$; for $k_1$ they follow from the fact that the Mittag-Leffler function $E_{\al,\be}(-x)\in\CMF$ for $x>0$,  $0\le\al\le 1, \be\ge \al$ (see e.g. \cite{Miller}) and Proposition 2.1(a) and (b). On the other hand, properties (d) and (e) are trivial for $g_1$, but need more elaborate proof for $g_2$.

Finally, consider the case of continuous distribution in its simplest form: constant weight function $\mu(\be)\equiv 1$. Then (\ref{h}) implies (taking $s^\be=e^{\be\log s}$)
$$   g_1(s)=\frac{s-1}{\log s},\ \ \ g_2(s)=\frac{s\log s}{s-1}.$$
Based on these explicit representations, only the positivity of the functions $g_i(s)$ for $s>0$ is evident, however properties (d) and (e) are not easily recognized. 

Now, we proceed with the proof of Theorem 3.1.

\begin{proof} 
Let us start with the kernel $k_2(t)$. 
Application of the inverse Laplace transform to $\widehat{k_2}(s)=h(s)/s$, see (\ref{ker}), implies by the use of (\ref{pairs}): 
\begin{equation}\label{k2}
k_2(t)=\int_0^1\mu(\be)\frac{t^{-\be}}{\Ga(1-\be)}\, d\be.
\end{equation}
In the case of discrete distribution, (\ref{k2}) and (\ref{hmt}) imply $$k_2(t)=\frac{t^{-\al}}{\Ga(1-\al)}+\sum_{j=1}^m b_j\frac{t^{-\al_j}}{\Ga(1-\al_j)}.$$ Therefore $$k_2(t)\sim t^{-\al},\ t\to 0;\ \ k_2(t)\sim t^{-\al_m},\ t\to \infty,$$ and thus (a) is satisfied for this kernel.

In the case of continuous distribution it is proven in \cite{K}, Proposition 2.1, that 
$$
 k_2(t)\sim \frac{1}{t(\log t)^2},\ t\to 0.
$$
Therefore, it is integrable at $t=0$ (note that the singularity at $t=0$ is quite strong). 
Moreover, since $\Gamma(1-\be)\ge 1$ for $\be\in[0,1]$, (\ref{k2}) implies for $t>1$
$$0\le k_2(t)\le \sup_{\be\in[0,1]}|\mu(\be)|\int_0^1t^{-\be}\,d\be\le C\frac{t-1}{t\log t}$$
and thus $k_2(t)\to 0$ as $t\to\infty$.
Complete monotonicity of $k_2(t)$ follows directly by differentiation of (\ref{k2}). In this way, (a) and (b) are proven for the kernel $k_2(t)$ in both discrete and continuous case.

 Consider now the kernel $k_1$. The identity $\widehat{k_1}(s)=1/h(s)$, see (\ref{ker}), implies the following representation for this kernel as an inverse Laplace integral:
\begin{equation}\label{L0}
  k_1(t)=\frac{1}{2\pi \mathrm{i}} \int_{\g-\mathrm{i}\infty}^{\g+\mathrm{i}\infty}e^{st}\frac{1}{h(s)}\,ds, \ \ \g>0.
  \end{equation}
	Consider first the discrete case in which $h(s)$ is defined by (\ref{hmt}).
The function 
$h(s)$ has no zeros in $\Si_\pi$. Indeed, for $s=r e^{i\phi}$, with $r>0$, $\phi\in(-\pi,\pi)$, $$\Im \{s^\al+\sum_{j=1}^m b_j s^{\al_j}\}=r^\al\sin\al\phi+\sum_{j=1}^m b_j r^{\al_j}\sin\al_j\phi\neq 0,$$ since $\sin\al\phi$ and $\sin\al_j\phi$ have the same sign and $b_j>0$. Then the function under the integral sign in (\ref{L0}) is analytic in $\Si_\pi$ and we can bend the integration contour to the contour 
 $\Gamma_{\rho,\theta}$ defined by 
\begin{equation}\label{Gamma}
\Gamma_{\rho,\theta}=\Gamma^-_{\rho,\theta}\cup \Gamma^0_{\rho,\theta} \cup\Gamma^+_{\rho,\theta}, \ \ \rho>0,\ \pi/2<\theta<\pi,
\end{equation}
 where
\begin{equation*}
   \Gamma^{\pm}_{\rho,\theta}=\left\{re^{\pm\mathrm{i}\theta}:\ r\geq \rho\right\},\  \Gamma^0_{\rho,\theta}= \left\{\rho e^{\mathrm{i}\psi}:
    \ |\psi|\le\theta\right\},
\end{equation*}
and $\Gamma_{\rho,\theta}$ is oriented in the direction of growth of $\arg s$. Hence
  \begin{equation}\label{L1}
  k_1(t)=\frac{1}{2\pi \mathrm{i}} \int_{\Gamma_{\rho,\theta}}e^{st}\frac{1}{s^\al+\sum_{j=1}^m b_j s^{\al_j}}\,ds. 
  \end{equation}
The integral over  $\Gamma^0_{\rho,\theta}$ is a function from $C^\infty[0,\infty)$. 
Take $\rho=R$ so large that
$$|s^\al+\sum_{j=1}^m b_j s^{\al_j}|\ge |s|^\al-\sum_{j=1}^m b_j |s|^{\al_j}\ge |s|^\al/2,\ \ |s|\ge R.$$
Then, noting that $\cos\theta<0$ for $\pi/2<\theta<\pi$, it follows
\begin{equation}\label{gamapm}
\left|\int_{\Gamma^-_{R,\theta}\cup \Gamma^+_{R,\theta}}e^{st}\frac{1}{s^\al+\sum_{j=1}^m b_j s^{\al_j}}\,ds\right|\le C\int_R^\infty e^{rt \cos\theta} r^{-\al}\, dr\le C t^{\al-1}.
\end{equation}
Therefore, $k_1(t)\sim t^{\al-1}$ for $t\to 0$ and thus it has an integrable singularity at $t=0$. 
Since in the discrete case $\widehat{k_1}(s)\sim s^{-\al_m}$ as $s\to 0$, Karamata-Feller Tauberian theorem (\cite{Feller}, Chapter XIII) implies $k_1(t)\sim t^{\al_m -1},\ \ t\to\infty$.
Thus (a) is proven for the discrete variant of $k_1(t)$.  To prove its complete monotonicity we take $\rho\to 0$ and $\theta\to \pi$ in (\ref{L1}). Since 
	\begin{equation}\label{gama0}
\left|\int_{\Gamma^0_{\rho,\theta}}e^{st}\frac{1}{s^\al+\sum_{j=1}^m b_j s^{\al_j}}\,ds\right|\le C\int_0^\pi e^{\rho t \cos\psi}\rho^{1-\al_m} \, d\psi ,
\end{equation} 
the integral over $\Gamma^0_{\rho,\theta}$ vanishes when $\rho\to 0$. Therefore, only the contributions of the integrals over $\Gamma^\pm_{\rho,\theta}$ remain in (\ref{L1}), implying
  \begin{equation}\label{GI}
  k_1(t)= \int_0^\infty e^{-rt}K(r)\,dr,
  \end{equation}
  where
  $$
  K(r)=-\frac{1}{\pi}\Im\left\{\left.\frac{1}{s^\al+\sum_{j=1}^m b_j s^{\al_j}}\right|_{s=r e^{\mathrm{i}\pi}}\right\}.
  $$
	Simplifying this expression, we get 
	$$K(r)=\frac1\pi\frac{B(r)}{(A(r))^2+(B(r))^2}$$ where
	$$A(r)=r^\al\cos\al\pi+\sum_{j=1}^m b_j r^{\al_j}\cos\al_j\pi,\ B(r)=r^\al\sin\al\pi+\sum_{j=1}^m b_j r^{\al_j}\sin\al_j\pi,$$
		and thus  $K(r)>0$ for $r>0$. This together with representation (\ref{GI}) implies that $k_1(t)\in\CMF$.

	In the case of continuous distribution it is proven in \cite{K}, Proposition 3.1, that for small values of $t$
$$
k_1(t)\le C \log\frac 1t,
$$
therefore this kernel has integrable singularity at $t\to 0$. Further, by \cite{K}, Proposition 2.2, (ii) and (iii),
$$
\widehat{k_1}(s)\sim \left(\log\frac{1}{s}\right)^{\la+1},\ \ s\to 0,
$$
where
$$\la=\left\{
\begin{array}{l}
0\ \  \mbox{ if}\ \mu(0)\neq 0,\\[6pt]
\nu\ \ \  \mbox{if}\ \mu(\be)=a\be^\nu\ \mbox{as}\ \be\to 0.
\end{array}
\right.
$$
Applying again Karamata-Feller Tauberian theorem (\cite{Feller}, Chapter XIII) it follows
$$
k_1(t)\sim \frac{(\log t)^\la}{t},\ \ t\to\infty, 
$$
and thus $k_1(t)\to 0$ as $t\to\infty$. Complete monotonicity of $k_1(t)$ in the case of continuous distribution is proven in \cite{K}, Propositions 3.1.   
	In this way the proof of properties (a) and (b) is completed for all cases.

According to (\ref{ker}) $$\widehat{k_1}(s)\widehat{k_2}(s)=1/s$$ and taking the inverse Laplace transform  of this identity we derive (c). 

Further, $g_1(s)\in\BBF$ and $g_1(s)/s\in\CMF$ by direct differentiation (see (\ref{g}) and (\ref{h})). To prove that $g_2(s)\in\BBF$ we use its representation (see (\ref{ker}) and (\ref{g})) $$g_2(s)=s \widehat {k_1}(s).$$ In view of the properties of $k_1(t)$ proven in (a) and (b) we can apply Proposition 2.1(d) which implies that $g_2(s)$ is a Bernstein function. Then by Proposition 2.1(c) $g_2(s)/s\in\CMF$. In this way statements (d) and (e) are proven. 

The first part of (f) follows from the complete monotonicity of $k_1(t)$ and $k_2(t)$, applying Proposition 2.1(e). 
The stronger mapping property of $g_i$ in the multi-term case follows by applying the inequalities $|\arg(s^\be)|= \be |\arg s|$, $|\arg(s_1+s_2)|\le \max \{|\arg s_1|,|\arg s_2|\}$ and $\arg(s^{-1})=\arg s$.
\end{proof}

Let us now prove that problem (\ref{C}) is equivalent to the abstract integral equation (\ref{V}) with kernel $k=k_1$ and problem (\ref{RL}) is equivalent to (\ref{V}) with kernel $k=k_2$. Following ideas from \cite {K}, we can 
rewrite the Caputo differential operator of distributed order 
$$^C\DD^{(\mu)}_t = \int_0^1 \mu(\be) ^C\! D_t^{\be}\,d\be $$
in the form 
\begin{equation}\label{DD}
^C \DD^{(\mu)}_t f=k_2\ast f'=(k_2\ast f)'-k_2(t)f(0),\ \ 
\end{equation}
for functions $f$ for which the expressions are well defined.  
Consider also the integral operator of distributed order for $f\in L^1$
\begin{equation}\label{JJ}
\JJ^{(\mu)}_t f=k_1\ast f.
\end{equation}
The functions $k_1$ and $k_2$ in (\ref{JJ}) and (\ref{DD}) are exactly the kernels, corresponding to problems (\ref{C}) and (\ref{RL}). Then 
\begin{equation}\label{DDJJ}
^C\DD^{(\mu)}_t \JJ^{(\mu)}_t f=f, \ \ \ \JJ^{(\mu)}_t {^C \DD^{(\mu)}_t} f=f-f(0).
\end{equation}
Indeed, (\ref{DD}), (\ref{JJ}) and Theorem 3.1.(c) imply
\begin{eqnarray}
&&^C \DD^{(\mu)}_t\JJ^{(\mu)}_t f=(k_2\ast (k_1\ast f))'-k_2(t)(k_1\ast f)(0)=((k_2\ast k_1)\ast f)'=(1\ast f )'=f,\nonumber\\
&& \JJ^{(\mu)}_t {^C \DD^{(\mu)}_t} f=k_1\ast(k_2\ast f')=(k_1\ast k_2)\ast f'=1\ast f '=f-f(0).\nonumber
\end{eqnarray}
Using identities (\ref{DDJJ}), we derive the Volterra integral equation (\ref{V}) with kernel $k=k_1$ by applying operator $\JJ^{(\mu)}_t$ to both sides of equation (\ref{C}). Conversely,  applying $^C \DD^{(\mu)}_t$ to both sides of (\ref{V}) with $k=k_1$, we get back equation (\ref{C}). 
The initial condition $u(0)=a$ is recovered by taking $t=0$ in (\ref{V}).

Concerning the Riemann-Liouville differential operator of distributed order $$\DD^{(\mu)}_t=\int_0^1 \mu(\be) D_t^{\be}\,d\be,$$ it  can be represented in the form
$$\DD^{(\mu)}_t f=(k_2\ast f)'.$$
This implies
$
J_t^1\DD^{(\mu)}_t f=k_2\ast f.
$
Then, the equivalence of problem (\ref{RL}) and the abstract integral equation (\ref{V}) with kernel $k=k_2$ can be established in a similar way as above, but applying the classical integration and differentiation operators $J_t^1$ and $D_t^1$ instead of  $\JJ^{(\mu)}_t$ and $^C \DD^{(\mu)}_t$.

\section{Well-posedness, representation formula and positivity }

The following definition is based on the equivalence proven at the end of the previous section.
\begin{definition}
Problem (\ref{C}), resp. (\ref{RL}), is said to be well-posed if the Volterra integral equation (\ref{V}) with kernel $k=k_1$, resp. $k=k_2$, is well-posed. The solution operator of the corresponding integral equation (\ref{V}) is called a solution operator of problem (\ref{C}), resp. (\ref{RL}).
\end{definition}
Now we formulate the main result of this paper.

\begin{theorem} 
Suppose that the conditions of Theorem 3.1. on the weight function $\mu(\be)$  are satisfied. Let  $A$ be a generator of a bounded $C_0$-semigroup $T(t)$, such that $\|T(t)\|\le M$, $t\ge 0$. 
Then problems (\ref{C}) and (\ref{RL}) are well-posed and their solution operators $S(t)$ satisfy $\|S(t)\|\le M$, $t\ge 0$. Moreover, the solutions of (\ref{C}) and (\ref{RL}) have the representation
\begin{equation}\label{repr}
u(t)=\lim_{n\to\infty}\frac{1}{n!}\left(n/t\right)^{n+1}\sum_{k=0}^n \sum_{p=1}^k b_{n,k,p}\left(n/t\right) \left(R\left(g\left(n/t\right),A\right)\right)^{p+1}a,
\end{equation}
where the convergence is uniform on bounded intervals of $t> 0$. The functions $b_{n,k,p}(s)$ are nonnegative for $s>0$ and are defined by
\begin{equation}\label{b}
b_{n,k,p}(s)=(-1)^{n+p}\bin{n}{k} \left(\frac{g(s)}{s}\right)^{(n-k)} a_{k,p}(s) p!,\ \ s>0,
\end{equation}
where the functions $a_{k,p}(s)$ are given by the recurrence relation
\begin{eqnarray}
&&a_{k+1,p}(s)=a_{k,p-1}(s)g'(s)+a_{k,p}'(s),\ \ 1\le p\le k+1,\ k\ge 1,\label{akp}\\
&&a_{k,0}=a_{k,k+1}\equiv 0,\ \ a_{1,1}(s)=g'(s).\nonumber
\end{eqnarray}
Here $g(s)=g_1(s)$ in case of problem (\ref{C}) and $g(s)=g_2(s)$ in case of problem (\ref{RL}). 
\end{theorem} 
\begin{proof}
First we prove (\ref{GenTh}), where $H(s)$ is defined by (\ref{H}) with $g(s)=g_i(s),\ i=1,2$. Let us express $H^{(n)}(s)$ in terms of powers of $$w(s)=R(g(s),A).$$ 
Note that by Theorem 3.1 if $s>0$ then $g(s)>0$ and thus $g(s)\in\ro(A)$, i.e. the resolvent operator $R(g(s),A)$ is well defined.
By the Leibniz rule it follows
\begin{equation}\label{Leibniz}
H^{(n)}(s)=\sum_{k=0}^n \bin{n}{k} \left(\frac{g(s)}{s}\right)^{(n-k)}w^{(k)}(s).
\end{equation}
The formula for the $k$-th derivative of a composite function (see \cite{Todorov}) gives
\begin{equation}\label{wk}
w^{(k)}(s)=\sum_{p=1}^k a_{k,p}(s)(-1)^p p! (R(g(s),A))^{p+1},
\end{equation}
where the functions $a_{k,p}(s)$ are defined by (\ref{akp}).

We will prove inductively that for any $k\ge 1$ and $1\le p\le k$  
\begin{equation}\label{CMF}
(-1)^{k+p}a_{k,p}(s)\in\CMF.
\end{equation}
 For $k=p=1$ this is fulfilled since $a_{1,1}(s)=g'(s)$ and $g'(s)\in\CMF$ by Theorem 3.1(d). 
Further, $a_{2,1}=g'',\ a_{2,2}=(g')^2$ and the assertion (\ref{CMF}) holds for these functions applying Theorem 3.1.(d) and Proposition 2.1.(a). Now fix some $k_0\ge 2$ and suppose that (\ref{CMF}) holds for all $k\le k_0$, $1\le p\le k$. Then, (\ref{akp}) implies that (\ref{CMF}) is satisfied for $k=k_0+1$, $1\le p\le k_0$,
since $(-1)^{k_0+p+1}a_{k_0,p-1}(s) g'(s)\in\CMF$ as a product of two completely monotone functions and $(-1)^{k_0+p+1}a_{k_0,p}'(s)\in\CMF$ by (\ref{CMF}). In addition, by (\ref{akp}), $a_{k_0+1,k_0+1}=a_{k_0,k_0}g'$ and it is completely monotone since $a_{k_0,k_0}\in\CMF$ and $g'\in\CMF$. In this way the proof of (\ref{CMF}) is completed. 

In particular, (\ref{CMF}) implies 
\begin{equation}\label{pos1}
(-1)^{k+p}a_{k,p}(s)\ge 0, \ \ s>0.
\end{equation}
On the other hand, by Theorem 3.1(e) $g(s)/s\in\CMF$,  i.e. 
\begin{equation}\label{pos2}
(-1)^{n-k}\left(g(s)/s\right)^{(n-k)}\ge 0,\ \ s>0.
\end{equation}

Inserting (\ref{wk}) in (\ref{Leibniz}) we get
\begin{equation}\label{Leibniz1}
(-1)^n H^{(n)}(s)=\sum_{k=0}^n \sum_{p=1}^k b_{n,k,p}(s) (R(g(s),A))^{p+1},
\end{equation}
where the functions $b_{n,k,p}(s)$ are defined in (\ref{b}). Moreover, inserting (\ref{pos1}) and (\ref{pos2}) in (\ref{b}),  it follows
\begin{equation}\label{posb}
b_{n,k,p}(s)\ge 0, \ s>0.
\end{equation}
In addition, let us note that in the trivial case $A\equiv 0$ (\ref{Leibniz1}) implies the identity
\begin{equation}\label{Leibniz2}
(-1)^n (s^{-1})^{(n)}=\sum_{k=0}^n \sum_{p=1}^k b_{n,k,p}(s) (g(s))^{-(p+1)}.
\end{equation}

Now, applying successively (\ref{posb}), (\ref{HY}) and (\ref{Leibniz2}) we obtain from (\ref{Leibniz1})
\begin{eqnarray}
\|H^{(n)}(s)\|&\le& \sum_{k=0}^n \sum_{p=1}^k b_{n,k,p}(s) \|(R(g(s),A))^{p+1}\|\nonumber\\
&\le& M \sum_{k=0}^n \sum_{p=1}^k b_{n,k,p}(s) ((g(s))^{-(p+1)}\nonumber\\
&=& M (-1)^n (s^{-1})^{(n)}=M n!s^{-(n+1)},\ \ s >0.\nonumber
\end{eqnarray}
Therefore, conditions (\ref{GenTh}) are satisfied and by the Generation Theorem problem (\ref{V}) is well-posed with bounded solution operators $S(t)$, satisfying $\|S(t)\|\le M$, $t\ge 0$. Therefore, by Definition 4.1 this holds for problems (\ref{C}) and (\ref{RL}).

Further, we use the Post-Widder inversion formula for the Laplace transform (see e.g. \cite{Laplace}): 
 Let $u(t),\ t\ge 0,$ be  a $X$ valued continuous function, such that $u(t)=O(e^{\g t})$ as $t\to\infty$ for some real $\g$. 
Then
$$
u(t)=\lim_{n\to\infty} \frac{(-1)^n}{n!}\left(\frac{n}{t}\right)^{n+1}\left(\frac{d^n \widehat{u}}{ds^n}\right)\left(\frac{n}{t}\right)
$$
 uniformly on compact subsets of $(0,\infty)$. 

Since $u(t)=S(t)a$ is a continuous and bounded function for $t\ge 0$, the Post-Widder inversion theorem can be applied and gives the representation (\ref{repr}).
\end{proof}

For an analogous result concerning the single-term fractional evolution equation see \cite{Baj}, Corollary 2.10. 

The representation (\ref{repr}) is a generalization of the exponential representation for the solution of the classical Cauchy problem (\ref{classical}) 
$$u(t)=\lim_{n\to\infty}\left(I-\frac{t}{n}A\right)^{-n}a.$$


The positivity of the coefficients $b_{n,k,p}$ in representation (\ref{repr}) has a useful direct consequence: it implies the positivity of the solution operator. This holds, however, only in spaces in which the notion of positivity is well-defined, i.e. in ordered Banach spaces (for a simple introduction see e.g. \cite{Engel}, p. 353).

 
Suppose $X$ is an ordered Banach space.
For example, such are the spaces of
type $L^p(\Om)$ or $C_0(\Om)$ for some $\Om\in\RR^d$, $d\in\NN$, with the
canonical ordering: 
a function $a \in X$
is positive (in symbols: $a\ge 0$) if $a(x)\ge 0$ for (almost) all $x\in\Om$.

A solution operator $S(t)$ in an ordered Banach space $X$ is called positive if 
$a\ge 0$ implies $S(t)a\ge 0$ for any $t\ge 0$.

In other words, positivity of a solution operator means that positivity of the initial condition is preserved in time. Next we prove that this is satisfied for the considered problems (\ref{C}) and (\ref{RL}) if the operator $A$ generates a positive $C_0$-semigroup.



\begin{corollary}
Let $X$ be an ordered Banach space. Assume the conditions of Theorem 4.2. are satisfied and the solution operator $T(t)$ of the classical Cauchy problem (\ref{classical}) is positive. Then the solution operators $S(t)$ of problems (\ref{C}) and (\ref{RL}) are positive.
\end{corollary}
\begin{proof}
Since
$$
R(s,A)=\int_0^\infty e^{-st}T(t)\,dt, \ \ s>0,
$$
the positivity of the $C_0$-semigroup $T(t)$  imply that the resolvent operator $R(s,A)$ is positive: if $a\in X$ and $a\ge 0$, then $R(s,A)a\ge 0$, $s> 0$. Therefore $R(g(s),A)a\ge 0$ for all $s> 0$. This together with the positivity of the coefficients (\ref{b}) in the representation formula (\ref{repr})  implies the positivity of $S(t)$.
\end{proof}

Positivity related to problem (\ref{C}) with operator $A=\De$ on $\RR^d$, $d\ge 1$, is established in \cite{Chechkin, Sokolov, K, Mirjana} by proving the positivity of the fundamental solution. 
In the case of a second order elliptic operator $A$ on a bounded domain, positivity is implied by a corresponding maximum principle, see e.g. \cite{LuchkoMultiterm}. 


\section{Subordination formula}

In the previous section we proved that well-posedness of the classical Cauchy problem (\ref{classical}) implies well-posedness of problems (\ref{C}) and (\ref{RL}). 
For completeness, here a subordination formula is given, relating the solution operator $S(t)$ of the fractional evolution equation (\ref{C}), resp. (\ref{RL}), and the solution operator $T(t)$ of the classical Cauchy problem (\ref{classical}).

\begin{theorem}
Assume the conditions of Theorem 3.1. on the weight function $\mu(\be)$  are satisfied and let $A$ be a generator of a bounded $C_0$-semigroup $T(t)$.
Then the solution operator $S(t)$ of problem (\ref{C}), resp.  (\ref{RL}), satisfies the subordination identity  
\begin{equation}\label{sub}
S(t)=\int_0^\infty \ph(t,\tau)T(\tau)\, d\tau, \ \ t>0,
\end{equation}
 with function $\ph(t,\tau)$ defined by
\begin{equation}\label{defphiint}
\ph(t,\tau)=\frac{1}{2\pi \mathrm{i}} \int_{\g-\mathrm{i}\infty}^{\g+\mathrm{i}\infty}e^{st-\tau g(s)}\,\frac{g(s)}{s}\,ds, \ \ \g,t,\tau>0,
\end{equation}
where $g(s)=g_1(s)$ in case of problem (\ref{C}) and $g(s)=g_2(s)$ in case of problem (\ref{RL}).
The function $\ph(t,\tau)$ is a probability density function, i.e. it satisfies the properties 
\begin{equation}\label{12}
\ph(t,\tau)\ge 0,\ \ \int_0^\infty \ph(t,\tau)\, d\tau=1.
\end{equation}
 Moreover, in the case of discrete distribution (\ref{case1}), there exists $\theta_0\in (0,\pi/2)$ such that $\ph(t,\tau)$ admits analytic extension to the sector $|\arg t|<\theta_0$ and is bounded on each subsector $|\arg t|\le\theta<\theta_0$.
\end{theorem}
\begin{proof}
According to (\ref{defphiint}) the Laplace transform of the function $\ph(t,\tau)$ with respect to $t$ 
$$
\widehat{\ph}(s,\tau)=\int_0^\infty e^{-st}\ph (t,\tau)\, dt,\ \ s,\tau>0,
$$
is given by
\begin{equation}\label{phi}
\widehat{\ph}(s,\tau)=\frac{g(s)}{s}e^{-\tau g(s)},\ \ s,\tau>0,
\end{equation}
Let $T(t)$ be the solution operator of the classical Cauchy problem (\ref{classical}) and define an operator-valued function $S(t)$ by (\ref{sub}). Application of the Laplace transform gives by using (\ref{phi})
$$
\int _0^\infty e^{-st}S(t)\, dt=\int_0^\infty\widehat{\ph}(s,\tau)T(\tau)\, d\tau=\frac{g(s)}{s}\int_0^\infty e^{-\tau g(s)}T(\tau)\, d\tau=\frac{g(s)}{s}R(g(s),A).
$$
Comparing this result to (\ref{H}), it follows by the uniqueness of the Laplace transform  that $S(t)$ is exactly the solution operator of (\ref{C}) if $g=g_1$, resp. (\ref{RL}) if $g=g_2$. In this way (\ref{sub}) is established. 

Let us prove now that $\ph(t,\tau)$ is a probability density function. Since $g(s)/s\in\CMF$ and $e^{-\tau g(s)}\in \CMF$ as a composition of the completely monotone in $x$ function $e^{-\tau x}$ and the Bernstein function $g(s)$, (\ref{phi}) implies that $\widehat{\ph}(s,\tau)\in\CMF$ as a product of two completely monotone functions (see Proposition 2.1(a),(b) and Theorem 3.1(d),(e)). Then Bernstein's theorem implies $\ph(t,\tau)\ge 0.$

The second identity in (\ref{12}) can be proven in various ways. Here we show that it is a particular case of (\ref{sub}). Indeed, in the trivial case $A\equiv 0$ the constant in $t$ function $u(t)\equiv a$ satisfies equations (\ref{C}), (\ref{RL}) and (\ref{classical}) and by the uniqueness of the solution it follows that this is the solution in this case, i.e.
$S(t)=T(t)=I$. Then (\ref{sub}) reduces to the desired identity.

 In \cite{K}, Section 4.2, it is proven that for $g(s)=g_1(s)$ in the case of continuous distribution the function (\ref{defphiint})
is well defined locally integrable function. In an analogous manner this can be done also for the continuous distribution variant of $g_2(s)$.
So, it remains to prove the last part of the theorem concerning the discrete distribution case. 
Define
\begin{equation}\label{theta}
\theta_0=\min\{(1/\al-1)\pi/2,\pi/2\}-\e,
\end{equation}
 where $\e>0$ is small enough, such that $\theta_0>0$. 
According to \cite{Pruss}, Theorem 0.1, it suffices to prove that the function $\widehat{\ph}(s,\tau)$ admits analytic extension to the sector $|\arg s|\le \pi/2+\theta_0$ and $s\widehat{\ph}(s,\tau)$ is bounded on each subsector $|\arg s|\le \pi/2+\theta,\ \theta<\theta_0$.  Theorem 3.1(f) states that $g(s)$ can be extended analytically to $\Si_\pi$, thus this holds also for the function $\widehat{\ph}(s,\tau)$. Take $s$ such that $|\arg s|\le \pi/2+\theta,\ \theta<\theta_0$. Then, applying again Theorem 3.1(f), it follows
$$
|\arg g(s)|\le\al |\arg s|<\pi/2-\al\e.
$$
Therefore, $g(s)=\rho e^{\mathrm{i}\phi}$, for some $\rho>0$, $|\phi|<\pi/2-\al\e$, and thus
$$
|s\widehat{\ph}(s,\tau)|=|g(s)e^{-\tau g(s)}|\le \rho e^{-\tau \rho \cos\phi}\le \rho e^{-a \rho }\le (ea)^{-1},
$$
where $a=\tau \sin{\al\e}>0$. 
With this the proof is completed.
\end{proof}

In the particular case of single term equation, such a subordination formula is given in \cite{Baj}, Theorem 3.1, where the function $\ph(t,\tau)$ is expressed in terms of a special function of Wright type.

In \cite{K} the subordination relation (\ref{sub}) is proven for the continuous distribution case of problem (\ref{C}) with $A$ being the Laplace operator.
  
Note that the subordination identity (\ref{sub}) shows again that in ordered Banach space positivity of $T(t)$ implies positivity of $S(t)$ .

\section{Conclusion}
An approach is proposed for the study of the distributed order fractional evolution equations in the Caputo and the Riemann-Liouville sense, rewriting them as abstract Volterra integral equations. The obtained results are based mainly on the properties of the kernels of these integral equations, and especially those related to complete monotonicity.

The results proven in this paper hold for a large class of fractional evolution equations, involving several types of fractional differentiation of order less than one as well as various possibilities for the operator $A$: e.g. the Laplace operator, general second order symmetric uniformly elliptic operators, operators leading to the so-called time-space fractional equations, such as: 
 space-fractional derivatives (e.g. in the Riesz sense), 
fractional powers of the multi-dimensional Laplace operator, other forms of fractional Laplacian (see e.g. \cite{Y2012, Y2010}), fractional powers of more general elliptic operators, etc. 

The developed technique is also applicable to other related problems, for example the Rayleigh-Stokes problem for the generalized second grade fluid with fractional derivative model, see e.g. \cite{BB, RS}, or to more general abstract Volterra integral equations with kernel $k(t)$, which Laplace transform $\widehat{k}(s)$ is well-defined for $s>0$ and is such that $(\widehat{k}(s))^{-1}$ is a Bernstein function.

\section{Acknowledgments}
The author is partially supported by Grant DFNI-I02/9/12.12.2014
from the Bulgarian National Science Fund and the Bilateral Research
Project "Mathematical modelling by means of integral transform methods, partial differential equations, special and
generalized functions" between BAS and SANU.



\end{document}